\def\DateTime{09/Octorber/2014, 19:00 (Kyoto)}
\def\Version{Version $1.1$}
\def\yes{\if00}
\def\no{\if01}
\def\iftenpt{\no}
\def\ifelevenpt{\no}
\def\iftwelvept{\yes}

\def\ifquery{\yes}

\iftenpt
\documentclass[leqno]{amsart}
\else\fi
\ifelevenpt
\documentclass[leqno,11pt]{amsart}
\else\fi
\iftwelvept
\documentclass[leqno,12pt]{amsart}
\else\fi

\usepackage{amssymb}
\usepackage{amscd}
\usepackage{verbatim}
\usepackage{mathrsfs}
\usepackage{color}

\usepackage[sc]{mathpazo} 
\usepackage[T1]{fontenc}

\ifelevenpt
\else\fi

\iftwelvept
\else\fi

\theoremstyle{plain}
\newtheorem{Theorem}{Theorem}[section]

\newtheorem{Lemma}[Theorem]{Lemma}

\newtheorem{Corollary}[Theorem]{Corollary}

\theoremstyle{definition}

\newtheorem{Question}[Theorem]{Question}

\newcommand{\ZZ}{{\mathbb{Z}}}

\newcommand{\RR}{{\mathbb{R}}}

\newcommand{\CC}{{\mathbb{C}}}
\newcommand{\PP}{{\mathbb{P}}}

\newcommand{\DD}{{\mathbb{D}}}
\newcommand{\OO}{{\mathcal{O}}}

\newcommand{\LLL}{{\mathscr{L}}}
\newcommand{\MMM}{{\mathscr{M}}}
\newcommand{\XXX}{{\mathscr{X}}}
\newcommand{\YYY}{{\mathscr{Y}}}

\newcommand{\Spec}{\operatorname{Spec}}

\newcommand{\acherncl}{\widehat{{c}}}
\newcommand{\adeg}{\widehat{\operatorname{deg}}}

\newcommand{\rest}[2]{\left.{#1}\right\vert_{{#2}}}
\ifquery
\def\query#1{\setlength\marginparwidth{65pt} %80pt for other size
\marginpar{\raggedright\fontsize{7.81}{9} 
\selectfont\upshape\hrule\smallskip 
#1\par\smallskip\hrule}}

\else
\def\query#1{}

\fi
\begin{document}

%%%%%%%%%%%
%% Title %%
%%%%%%%%%%%
\title[Semiample invertible sheaves]%
{Semiample invertible sheaves with semipositive continuous hermitian metrics}
\author{Atsushi Moriwaki}
\address{Department of Mathematics, Faculty of Science,
Kyoto University, Kyoto, 606-8502, Japan}
\email{moriwaki@math.kyoto-u.ac.jp}
\date{\DateTime, (\Version)}
%\thanks{}
\begin{abstract}
Let $(L, h)$ be a pair of a semiample invertible sheaf and a semipositive continuous hermitian metric on a proper algebraic variety. 
In this paper, we prove that $(L, h)$ is semiample metrized, which is a generalization of the question due to S. Zhang. 
\end{abstract}

\maketitle

\section*{Introduction}

Let $X$ be a proper algebraic variety over $\CC$.
Let $L$ be an invertible sheaf on $X$ and let $h$ be a continuous hermitian metric of $L$.
We say that $(L, h)$ is {\em semiample metrized} if,
for any $\epsilon > 0$, there is $n > 0$ such that,
for any $x \in X(\CC)$, we can find $l \in H^0(X, L^{\otimes n}) \setminus \{ 0\}$ with
\[
\sup \left\{ h^{\otimes n}(l,l)(w) \mid w \in X(\CC) \right\}  \leq e^{\epsilon n} h^{\otimes n}(l,l)(x).
\]
In the paper \cite{ZhPos}, Shouwu Zhang proposed the following question:

\begin{Question}{\cite[Question~3.6]{ZhPos}}
If $L$ is ample and $h$ is smooth and semipositive, then does it follow that $(L, h)$ is semiample metrized?
\end{Question}

In \cite[Theorem~3.5]{ZhPos},  he actually gave the affirmative answer in the case where $X$ is smooth over $\CC$.
The purpose of this paper is to give an answer for a generalization of the above question.
First of all, we fix notations:
We say that $L$ is {\em semiample} if there is a positive integer $n_0$ such that $L^{\otimes n_0}$ is generated by
global sections.  Moreover, $h$ is said to be {\em semipositive } (or we say that $(L, h)$ is semipositive) if,
for any point $x \in X(\CC)$ and a local basis $s$ of $L$ on a neighborhood of $x$,
$-\log h(s,s)$ is plurisubharmonic around $x$ (for the definition of plurisubharmonicity on a singular variety,
see Section~\ref{sec:plurisubharmonic}).
Note that $h$ is not necessarily smooth. 
By using the recent work \cite{CGZ} due to Coman, Guedj and Zeriahi,
we have the following answer:

\begin{Theorem}
\label{thm:answer:Zhang:question}
If $L$ is semiample and $h$ is continuous and semipositive,
then $(L, h)$ is semiample metrized.
\end{Theorem}

Finally I would like to thank Prof. Zhang for his comments and suggestions.

\section{Plurisubharmonic functions  on singular complex analytic spaces}
\label{sec:plurisubharmonic}
Let $T$ be a reduced complex analytic space.
An upper-semicontinuous function 
\[
\varphi : T \to \RR \cup \{-\infty\}
\]
is said to be {\em plurisubharmonic} if
$\varphi \not\equiv-\infty$ and, for each $x \in T$,
there is an open neighborhood $U_x$ of $x$
together with
an open set $W_{x}$ of $\CC^{n_{x}}$ and a plurisubharmonic function $\Phi_{x}$ on $W_{x}$
such that $U_{x}$ is a closed complex analytic subspace of $W_{x}$ and
$\rest{\varphi}{U_{x}} = \rest{\Phi_{x}}{U_{x}}$.
For an analytic map $f : T' \to T$ of reduced complex analytic spaces and 
a plurisubharmonic function $\varphi$ on $T$,
it is easy to see that
$\varphi \circ f$ is either identically $-\infty$ or plurisubharmonic on $T'$.
By the theorem due to  Fornaess and Narasimhan \cite[Theorem~5.3.1]{FN},
an upper-semicontinuous function $\varphi : T \to \RR \cup \{-\infty\}$ is plurisubharmonic if and only if,
for any analytic map $\varrho : \DD \to T$, 
$\varphi \circ \varrho$ is either identically $-\infty$ or subharmonic on $\DD$,
where $\DD := \{ z \in \CC \mid \vert z \vert < 1 \}$.
Moreover, if $T$ is compact and $\varphi$ is plurisubharmonic on $T$, then
$\varphi$ is locally constant.

Let $\omega$ be a smooth $(1,1)$-form on $T$ such that, for each $x \in T$,
$\omega$ is locally given by $dd^c(u)$ for some smooth function $u$ on a neighborhood of $x$.
Let $\phi$ be a {\em quasi-plurisubharmonic function} on $T$, that is,
for each $x \in T$, $\phi$ can be locally written by the sum of a smooth function and
a plurisubharmonic function around $x$.
We say that $\phi$ is {\em $\omega$-plurisubharmonic}
if there is an open covering $T = \bigcup_{\lambda} U_{\lambda}$ together with
a smooth function $u_{\lambda}$ on $U_{\lambda}$ for each $\lambda$ such that $\rest{\omega}{U_{\lambda}} = dd^c(u_{\lambda})$ and
$\rest{\phi}{U_{\lambda}} + u_{\lambda}$ is plurisubharmonic on $U_{\lambda}$.
The condition for $\omega$-plurisubharmonicity is often denoted by $dd^c([\phi]) + \omega \geq 0$.

Here we consider the following lemma.

\begin{Lemma}
\label{lem:criterion:cont:psh}
Let $f : X \to Y$ be a surjective and proper morphism of algebraic varieties over $\CC$.
Let $\varphi$ be a real-valued function on  $Y(\CC)$. Then we have the following:
\begin{enumerate}
\renewcommand{\labelenumi}{(\arabic{enumi})}
\item
$\varphi$ is continuous if and only if $\varphi \circ f$ is continuous.

\item
We assume that $\varphi$ is continuous.
Then $\varphi$ is plurisubharmonic  if and only if $\varphi \circ f$ is  plurisubharmonic.
\end{enumerate}
\end{Lemma}

\begin{proof}
(1) It is sufficient to see that if
$\varphi \circ f$ is continuous, then $\varphi$ is continuous.  Otherwise,
there are $y \in Y(\CC)$, $\epsilon_0 > 0$ and a sequence $\{ y_n \}$ on $Y(\CC)$ such that
$\lim_{n\to\infty} y_n = y$ and $\vert \varphi(y_n) - \varphi(y) \vert \geq \epsilon_0$ for all $n$.
We choose $x_n \in X(\CC)$ such that $f(x_n) = y_n$.
As $f : X \to Y$ is proper, we can find a subsequence $\{ x_{n_i} \}$ of $\{ x_n \}$ such that
$x := \lim_{i\to\infty} x_{n_i}$ exists in $X(\CC)$.
Note that 
\[
f(x) = \lim_{i\to\infty} f(x_{n_i}) =  \lim_{i\to\infty} y_{n_i} = y,
\]
so that, as $\varphi \circ f$ is continuous,
\[
\varphi(y) = (\varphi \circ f)(x)  = \lim_{i\to\infty} (\varphi \circ f)(x_{n_i})  = \lim_{i\to\infty} \varphi(f(x_{n_i})) = \lim_{i\to\infty} \varphi(y_{n_i}),
\]
which is a contradiction, so that $\varphi$ is continuous.

\medskip
(2) We need to check that if $\varphi \circ f$ is  plurisubharmonic, then $\varphi$ is  plurisubharmonic.
By using Chow's lemma,  we may assume that $f : X \to Y$ is projective.
Moreover, since the assertion is local with respect to $Y$,
we may further assume that there is a closed embedding $\iota : X \hookrightarrow Y \times \PP^N$ such that
$p \circ \iota = f$, where $p : Y \times \PP^n \to Y$ is the projection to the first factor.
The remaining proof is  same as the last part of the proof of \cite[Theorem~1.7]{DM}.
Let $g : (\DD, 0) \to (Y, y)$ be a germ of analytic map.
By the theorem due to  Fornaess and Narasimhan, it is sufficient to show that $\varphi \circ g$ is
subharmonic.
Clearly we may assume that $g$ is given by the normalization of
a $1$-dimensional irreducible germ $(C, y)$ in $(Y,y)$.
Using hyperplanes in $\PP^N$, we can find $x \in X$ and a $1$-dimensional  irreducible germ $(C', x)$ in $(X,x)$
such that $(C', x)$ is lying over $(C, y)$.
Let $g' : (\DD, 0) \to (X, x)$ be the germ of analytic map given by the normalization of $(C', x)$.
Then we have the analytic map  $\sigma : (\DD, 0) \to (\DD, 0)$ with
$g \circ \sigma =  f \circ g'$:
\[
\begin{CD}
(\DD, 0) @>{g'}>> (X, x) \\
@V{\sigma}VV @VV{f}V \\
(\DD, 0) @>{g}>> (Y, y)
\end{CD}
\]
Changing a variable of $(\DD, 0)$, we may assume that $\sigma$ is given by $\sigma(z) = z^m$ for some positive integer $m$.
Then $\varphi \circ g \circ \sigma$ is subharmonic because
$\varphi \circ f$ is plurisubharmonic.
Therefore,  as $\sigma$ is \'{e}tale over the outside of $0$,
$\varphi \circ g$ is subharmonic on the outside of $0$,  and hence
$\varphi \circ g$ is subharmonic on $(\DD, 0)$ by removable singularities of subharmonic functions.
\end{proof}

\section{Descent of semipositive continuous hermitian metric}
 
In this section, we consider a descent problem of a semipositive continuous hermitian metric.

\begin{Theorem}
\label{thm:descent:semipositive:metric}
Let $f : X \to Y$ be a surjective and proper morphism of algebraic varieties over $\CC$ with $f_*\OO_X = \OO_Y$.
Let $L$ be an invertible sheaf on $Y$.
If $h'$ is a semipositive continuous hermitian metric of $f^*(L)$, then
there is a semipositive continuous hermitian metric $h$ of $L$ such that
$h' = f^*(h)$.
\end{Theorem}

\begin{proof}
Let $h_0$ be a continuous hermitian metric of $L$ on $Y$.
There is a continuous function $\phi$ on $X(\CC)$ such that $h' = \exp(\phi) f^*(h_0)$.
Let $F$ be a  subvariety of $X$  such that $F$ is an irreducible component of a fiber of $f : X \to Y$.
Then, as
\[
\rest{(f^*(L), h')}{F} \simeq (\OO_{F}, \exp(\rest{\phi}{F})),
\]
we can see that $-\rest{\phi}{F}$ is plurisubharmonic, so that
$\rest{\phi}{F}$ is constant.
Therefore, for any point $y \in Y(\CC)$, $\rest{\phi}{\mu^{-1}(y)}$ is constant
because $\mu^{-1}(y)$ is connected, and hence
there is a function $\psi$ on $Y(\CC)$ such that $\psi  \circ f = \phi$.
By (1) in Lemma~\ref{lem:criterion:cont:psh}, $\psi$ is continuous, so that,
if we set $h := \exp(\psi)h_0$, then
$h$ is continuous on $Y(\CC)$ and $h' = f^*(h)$.

Finally let us see that $h$ is semipositive. 
As it is a local question on $Y$,
we may assume that there is a local basis $s$ of $L$ over $Y$.
If we set $\varphi = - \log h(s,s)$, then
$\varphi \circ f$ is plurisubharmonic because $h'$ is semipositive.
Therefore, by (2) in Lemma~\ref{lem:criterion:cont:psh},
$\varphi$ is plurisubharmonic, as required
\end{proof}

\section{The proof of Theorem~\ref{thm:answer:Zhang:question}}
\renewcommand{\theequation}{\arabic{section}.\arabic{Claim}}

In the case where $X$ is smooth over $\CC$, $L$ is ample and $h$ is smooth,
this theorem was proved by Zhang \cite[Theorem~3.5]{ZhPos}.
First we assume that $L$ is ample. Then there are a positive integer $n_0$ and 
a closed embedding $X \hookrightarrow \PP^N$ such that $\rest{\OO_{\PP^N}(1)}{X} \simeq L^{\otimes n_0}$.
Let $h_{FS}$ be the Fubini-Study metric of $\OO_{\PP^n}(1)$.
Let $\phi$ be the continuous function on $X(\CC)$ given by $h^{\otimes n_0} = \exp(-\phi)\rest{h_{FS}}{X}$.
We set $\omega = c_1(\OO_{\PP^N}(1), h_{FS})$.
Then $\phi$ is $\left(\rest{\omega}{X}\right)$-plurisubharmonic.
Therefore, by \cite[Corollary~C]{CGZ},
there is a sequence $\{ \varphi_i \}$ of smooth functions on $\PP^N(\CC)$ with the following properties:
\begin{enumerate}
\renewcommand{\labelenumi}{(\arabic{enumi})}
\item
$\varphi_i$ is $\omega$-plurisubharmonic for all $i$.

\item $\varphi_i \geq \varphi_{i+1}$ for all $i$.

\item For $x \in X(\CC)$, $\lim_{i\to\infty} \varphi_i(x) = \phi(x)$.
\end{enumerate}
Since $X$ is compact and $\phi$ is continuous, (3) implies that the sequence $\{ \varphi_i \}$ converges $\phi$ uniformly on $X(\CC)$.
We choose $i$ such that $| \phi(x) - \varphi_i(x) | \leq \epsilon n_0/2$ for all $x \in X$.
We set $h_i = \exp(-\varphi_i) h_{FS}$.
Then $h_i$ is a semipositive smooth hermitian metric of $\OO_{\PP^N}(1)$.
Therefore, there is a positive integer $n_1$ such that,
for $x \in \PP^N(\CC)$, we can find $l \in H^0(\PP^N, \OO_{\PP^N}(n_1)) \setminus \{ 0 \}$ with
\[
 \sup \left\{ h_i^{\otimes n_1}(l,l)(w) \mid w \in \PP^N(\CC) \right\}
 \leq e^{n_1(\epsilon n_0 /2)} h_i^{\otimes n_1}(l,l)(x).
\]
In particular, if $x \in X(\CC)$, then $l(x) \not= 0$ (so that $\rest{l}{X} \not= 0$) and
\[
\sup \left\{ h_i^{\otimes n_1}(l,l)(w) \mid w \in X(\CC) \right\}
 \leq e^{\epsilon n_0 n_1/2} h_i^{\otimes n_1}(l,l)(x).
\]
Note that 
\addtocounter{Claim}{1}
\begin{equation}
\label{eqn:thm:answer:Zhang:question:01}
h^{\otimes n_0} e^{-\epsilon n_0/2} \leq   h_i  \leq h^{\otimes n_0}
\end{equation}
on $X(\CC)$ because 
$h_i = h^{\otimes n_0}  \exp(\phi - \varphi_i)$ and
$-\epsilon n_0/2 \leq \phi - \varphi_i \leq 0$
on  $X(\CC)$.
Therefore,
\[
\sup \left\{ h^{\otimes n_0 n_1}(l,l)(w) \mid w \in X(\CC) \right\} e^{-n_0 n_1\epsilon/2} 
\leq \sup \left\{ h_i^{\otimes n_1}(l,l)(w) \mid w \in X(\CC) \right\}
\]
and
\[
h_i^{\otimes n_1}(l, l)(x) \leq h^{\otimes n_0 n_1}(l,l)(x),
\]
and hence
\[
\sup \left\{ h^{\otimes n_0 n_1}(l,l)(w) \mid w \in X(\CC) \right\} 
\leq e^{n_1n_0 \epsilon} h^{\otimes n_0n_1}(l, l)(x),
\]
as required.

In general, as $L$ is semiample,
there are a positive integer $n_2$, a projective algebraic variety $Y$ over $\CC$,
a morphism $f : X \to Y$ and an ample invertible sheaf $A$ on $Y$ such that
$f_*\OO_X = \OO_Y$ and $f^*(A) \simeq L^{\otimes n_2}$.
Thus, by Theorem~\ref{thm:descent:semipositive:metric},
there is a semipositive continuous hermitian metric $k$ of $A$ such that
$(f^*(A), f^*(k)) \simeq (L^{\otimes n_2}, h^{\otimes n_2})$. 
Therefore, the assertion of the theorem follows from the previous observation.

\section{A variant of Theorem~\ref{thm:answer:Zhang:question}}
The following theorem is a consequence of  Theorem~\ref{thm:answer:Zhang:question} together with
the arguments in \cite[Theorem~3.3]{ZhPos}.
However, we can give a direct proof using ideas in the proof of Theorem~\ref{thm:answer:Zhang:question}.

\begin{Theorem}
\label{thm:variant}
Let $X$ be a projective algebraic variety over $\CC$.
Let $L$ be an ample invertible sheaf on $X$ and let $h$ be a semipositive continuous hermitian metric of $L$.
Let us fix a reduced subscheme $Y$ of $X$, $l \in H^0(Y, \rest{L}{Y})$ and a positive number $\epsilon$.
Then,  for  the given $X$, $L$, $h$, $Y$, $l$ and $\epsilon$,
there is a positive integer $n_1$ such that, for all $n \geq n_1$,
we can find $l' \in H^0(X, L^{\otimes n})$ with $\rest{l'}{Y} = l^{\otimes n}$ and
\[
\sup \left\{ h^{\otimes n}(l', l')(w) \mid w \in X(\CC) \right\} \leq
e^{n\epsilon} \sup \left\{ h(l, l)(w) \mid w \in Y(\CC) \right\}^n.
\]
\end{Theorem}

\begin{proof}
In the case where $X$ is smooth over $\CC$ and $h$ is smooth and positive, 
the assertion of the theorem follows from \cite[Theorem~2.2]{ZhPos}, in which
$Y$ is actually assumed to be a subvariety of $X$. However, the proof works well under the assumption that
$Y$ is a reduced subscheme. First of all, let us see the theorem in the case where $X$ is smooth over $\CC$ and 
$h$ is smooth and semipositive.
As $L$ is ample, there is a positive smooth hermitian metric $t$ of $L$ with $t \leq h $.
Let us choose a positive integer $m$ such that $e^{-\epsilon/2} \leq (t/h)^{1/m} \leq 1$ on $X(\CC)$.
If we set $t_m = h^{1 - 1/m}t^{1/m}$, then $t_m$ is smooth and positive, so that,
for a sufficiently large integer $n$,
there is $l' \in H^0(X, L^{\otimes n})$ such that $\rest{l'}{Y} = l^{\otimes n}$ and
\[
\sup \left\{ t_m^{\otimes n}(l', l')(w) \mid w \in X(\CC) \right\} \leq
e^{n\epsilon/2} \sup \left\{ t_m(l, l)(w) \mid w \in Y(\CC) \right\}^n,
\]
and hence the assertion follows because $e^{-\epsilon/2}h \leq t_m \leq h$ on $X(\CC)$.

For a general case, 
we use the same symbols $n_0$, $X \hookrightarrow \PP^N$, $h_{FS}$, $\phi$, $\omega$ and
$\{ \varphi_i \}$ as in the proof of Theorem~\ref{thm:answer:Zhang:question}.
Clearly we may assume that $l \not= 0$. Since $L$ is ample,
if $a_0$ is sufficiently large integer, then, for each $j=0, \ldots, n_0 - 1$, 
there is $l_j \in H^0(X, L^{\otimes n_0 a_0 + j})$ with $\rest{l_j}{Y} = l^{\otimes n_0 a_0 + j}$.
Let us fix a positive number $A$ such that 
\addtocounter{Claim}{1}
\begin{multline}
\label{eqn:thm:variant:01}
\sup \left\{ h^{\otimes n_0 a_0 + j}(l_j, l_j)(w) \mid w \in X(\CC) \right\} \\
\leq
e^{A} \sup \left\{ h(l, l)(w) \mid w \in Y(\CC) \right\}^{n_0 a_0 + j}
\end{multline}
for $j = 0, \ldots, n_0 - 1$. 
We choose $i$ with $| \phi(x) - \varphi_i(x) | \leq \epsilon n_0/2$ for all $x \in X$, and 
we set $h_i = \exp(-\varphi_i) h_{FS}$. 
As $h_i$ is smooth and semipositive, for the given $\PP^N$, $\OO_{\PP^N}(1)$,
$h_i$, $Y$, $l^{\otimes n_0}$ (as an element of $H^0(Y, \rest{\OO_{\PP^N}(1)}{Y})$) and $n_0\epsilon/4$,
there is a positive integer $a_1$ such that the assertion of the theorem holds for all $a \geq a_1$.
We put 
\[
n_1 := n_0 \max\left\{ a_1 + a_0 + 1,\ \frac{4A}{n_0\epsilon} -3a_0 + 1 \right\}.
\]
Let $n$ be an integer with $n \geq n_1$.
If we set $n = n_0 (a+a_0) + j$ ($0 \leq j \leq n_0 - 1$),
then 
\[
a \geq a_1\quad\text{and}\quad a \geq \frac{4A}{n_0\epsilon} - 4a_0,
\]
so that we can find $l'' \in H^0(\PP^N, \OO_{\PP^N}(a))$ with $\rest{l''}{Y} = l^{\otimes n_0 a}$ and
\begin{multline*}
\sup \left\{ h_i^{\otimes a}(l'', l'')(w) \mid w \in \PP^N(\CC) \right\} \\
\leq e^{a (n_0 \epsilon/4)} \sup \left\{ h_i(l^{\otimes n_0}, l^{\otimes n_0})(w) \mid w \in Y(\CC) \right\}^a,
\end{multline*}
which implies
\addtocounter{Claim}{1}
\begin{multline}
\label{eqn:thm:variant:02}
\sup \left\{ h^{\otimes n_0 a}(l'', l'')(w) \mid w \in X(\CC) \right\} \\
\leq
e^{(3/4) n_0 a \epsilon} \sup \left\{ h(l, l)(w) \mid w \in Y(\CC) \right\}^{n_0 a}
\end{multline}
because of \eqref{eqn:thm:answer:Zhang:question:01}. Here we set $l' = l'' \otimes l_j$.
Then, $\rest{l'}{Y} = l^{\otimes n}$ and, using \eqref{eqn:thm:variant:01} and \eqref{eqn:thm:variant:02}, we have
\begin{multline*}
\sup \left\{ h^{\otimes n}(l', l')(w) \mid w \in X(\CC) \right\} \\
\leq 
\sup \left\{ h^{\otimes n_0 a}(l'', l'')(w) \mid w \in X(\CC) \right\} \sup \left\{ h^{\otimes n_0 a_0 + j}(l_j, l_j)(w) \mid w \in X(\CC) \right\} \\
 \leq e^{(3/4) n_0 a \epsilon + A} \sup \left\{ h(l, l)(w) \mid w \in Y(\CC) \right\}^{n},
\end{multline*}
which implies the assertion 
because $(3/4) n_0 a \epsilon + A \leq \epsilon n$.
\end{proof}

\section{Arithmetic application}
As an application of Theorem~\ref{thm:answer:Zhang:question}, 
we have the following generalization of arithmetic Nakai-Moishezon's criterion (c.f. \cite[Corollary~4.8]{ZhPos}).

\begin{Corollary}
\label{thm:arith:Nakai:Moishzon}
Let $\XXX$ be a projective and flat integral scheme over $\ZZ$.
Let $\LLL$ be an invertible sheaf on $\XXX$ such that $\LLL$ is nef on every fiber of $\XXX \to \ZZ$.
Let $h$ be an $F_{\infty}$-invariant semipositive continuous hermitian metric of $\LLL$,
where $F_{\infty}$ is the complex conjugation map $\XXX(\CC) \to \XXX(\CC)$.
If $\adeg(\acherncl_1(\rest{(\LLL, h)}{\YYY})^{\dim \YYY}) > 0$ for all horizontal integral subschemes $\YYY$ of $\XXX$,
then, for an $F_{\infty}$-invariant continuous hermitian invertible sheaf $(\MMM, k)$ on $\XXX$,
$H^0(\XXX, \LLL^{\otimes n} \otimes \MMM)$ has a basis consisting of  strictly small sections
for a sufficiently large integer $n$.
\end{Corollary}

\begin{proof}
Let $X$ be the generic fiber of $\XXX \to \Spec(\ZZ)$ and let $Y$ be a subvariety of $X$.
Let $\YYY$ be the Zariski closure of $Y$ in $\XXX$.
As 
\[
\adeg(\acherncl_1(\rest{(\LLL, h)}{\YYY})^{\dim \YYY}) > 0,
\]
$\rest{(\LLL, h)}{\YYY}$ is big by \cite[Theorem~6.6.1]{MoArZariski}, so that 
$H^0(\YYY, \rest{\LLL^{\otimes n_0}}{\YYY}) \setminus \{ 0 \}$ has a strictly small section for
a sufficiently large integer $n_0$.
Moreover, if we set $L = \rest{\LLL}{X}$, then $\rest{L}{Y}$ is big, and hence $\deg(L^{\dim Y} \cdot Y) >0$ because $L$ is nef.
Therefore, $L$ is ample by Nakai-Moishezon's criterion for ampleness.
In particular, by Theorem~\ref{thm:answer:Zhang:question}, $h$ is semiample metrized.
Thus the assertion follows from the arguments in \cite[Theorem~4.2]{ZhPos}.
\end{proof}


\begin{thebibliography}{99}
\bibitem{CGZ}
D. Coman, V. Guedj and A. Zeriahi,
Extension of plurisubharmonic functions with growth control,
J. reine angew Math., {\bf 676} (2013), 33--49.

\bibitem{DM}
J.-P. Demailly,
Mesures de Monge-Amp\`{e}re et caract\'{e}risation g\'{e}om\'{e}trique
des vari\'{e}t\'{e}s alg\'{e}briques affines,
M\'{e}moires de la S. M. F., {\bf 19} (1985), 1-125.


\bibitem{FN}
J. E. Fornaess and R. Narasimhan,
The Levi problem on complex space with singularities,
Math. Ann., {\bf 248} (1980), 47--72.

\bibitem{MoArZariski}
A. Moriwaki,
Zariski decompositions on arithmetic surfaces, 
Publ. Res. Inst. Math. Sci. {\bf 48} (2012), 799-898.

\bibitem{ZhPos}
S. Zhang, Positive line bundles on arithmetic varieties,
J. of AMS, {\bf 8} (1995), 187-221.
\end{thebibliography}
\end{document}